\documentclass{amsart}

\usepackage{amscd}
\usepackage{enumitem}

\newtheorem{theorem}{Theorem}[section]
\newtheorem{lemma}[theorem]{Lemma}

\theoremstyle{definition}
\newtheorem{definition}[theorem]{Definition}
\newtheorem{example}[theorem]{Example}

\newtheorem{corollary}[theorem]{Corollary}
\newtheorem{proposition}[theorem]{Proposition}
\theoremstyle{remark}
\newtheorem{remark}[theorem]{Remark}

\numberwithin{equation}{section}

\begin{document}

\title{On Lattice points and Optimal Packing of Minkowski's Balls and Domains}
 
\author{Nikolaj M. Glazunov}
\address{Glushkov Institute of Cybernetics NASU, Kiev, Ukraine}
\curraddr{Institute of Mathematics and Informatics Bulgarian Academy of Sciences, 1113 Sofia, Bulgaria}
\email{glanm@yahoo.com}
\thanks{The author was supported   Simons grant 992227.}






\begin{abstract}
 
We investigate lattice packings of  Minkowski's balls  and   domains, as well as the distribution of lattice points on  Minkowski's curves which are boundaries of  Minkowski's balls. By results of the proof of Minkowski's conjecture about the critical determinant we devide the balls and domains on 3 classes: Minkowski, Davis  and  Chebyshev-Cohn balls.  The optimal lattice packings of the  balls and domains are obtained.  The minimum areas of  hexagons inscribed in the balls and domains and  circumscribed around their are given. We construct direct systems of these balls, domains and their critical lattices and calculate their direct limits.
\end{abstract}

\maketitle

\section*{Introduction}
  Let
\begin{equation}
\label{mb}
   D_p: \;  |x|^p + |y|^p \le 1, \; p \ge 1
\end{equation}
be (two-dimensional) Minkowski's balls with boundary 
\begin{equation}
C_p: \; |x|^p + |y|^p = 1, \;  p \ge 1.
\end{equation}
\
By results of the proof of Minkowski's conjecture we divide balls (\ref{mb}) and for natural $m$ corresponding domains $2^m D_p,$ on three classes: Minkowski, Davis and  Chebyshev-Cohn balls and domains (see below).

In the paper we prove two main theorems.

  \begin{theorem} (cf. Theorem \ref{opb})
The optimal lattice packing of  Minkowski and Chebyshev-Cohn balls is realized with respect to the sublattices of the index two of  critical lattices
  \[
  (-2^{-1/p},2^{-1/p}) \in \Lambda_{p}^{(1)}.
  \]   
The dencity of this optimal packing is given by the expression
\[
    V(D_p)/\Delta^{(1)}_p (2D_p),
\]
where $V(D_p)$ is the volume of the ball $D_p$ and $\Delta^{(1)}_p (2D_p)$ is the critical determinant of the domain $2D_p$.\\
The optimal lattice packing of   Davis  balls is realized with respect to the sublattices of the index two of  critical lattices
  \[
  (1,0)\in\Lambda_{p}^{(0)},
  \]   
The dencity of this optimal packing is given by the expression
\[
    V(D_p)/\Delta^{(0)}_p(2D_p),
\]
where $V(D_p)$ is the volume of the ball $D_p$ and $\Delta^{(0)}_p(2D_p)$ is the critical determinant of the domain $2D_p$.
  \end{theorem}

\begin{theorem} (cf. Theorem \ref{opd})
The optimal lattice packing of  Minkowski and Chebyshev-Cohn domains $2^m D_p$ is realized with respect to the sublattices of the index $2^{m+1}$ of  critical lattices
  \[
  (-2^{-1/p},2^{-1/p}) \in \Lambda_{p}^{(1)}.
  \]   
The dencity of this optimal packing is given by the expression
\[
    V(2^m D_p)/\Delta^{(1)}_p (2^{m+1}D_p),
\]    
    where $V(2^m D_p)$ is the volume of the domainl $2^m D_p$ and $\Delta^{(1)}_p (2^{m+1} D_p)$ is the critical determinant of the domain $2^{m+1} D_p$.\\
The optimal lattice packing of  Davis domains $2^m D_p$ is realized with respect to the sublattices of the index $2^{m+1}$  of  critical lattices
  \[
  (1,0)\in\Lambda_{p}^{(0)},
  \]   
The dencity of this optimal packing is given by the expression 
\[
    V(2^m D_p )/\Delta^{(0)}_p(2^{m+1} D_p)
\]
 where $V(2^m D_p)$ is the volume of the domainl $2^m D_p$ and $\Delta^{(0)}_p (2^{m+1} D_p)$ is the critical determinant of the domain $2^{m+1} D_p$.\\
     \end{theorem}

A system of equal  balls in $n$-dimensional space is said to form a packing, if no two balls of the 
system have any inner points in common.
Recently, remarkable results in resolving the problem of optimal packing of Euclidean balls in $8$ and 
$24$-dimensional real Euclidean spaces have been obtained by Viazovska \cite{via} and Cohn, Kumar, Miller, Radchenko, and Viazovska \cite{ckmrv},
respectively. In this research,  acknowledged by a Fields medal, optimal packings of Euclidean balls are proved to be defined by famous lattices.
Our investigations connect with Minkowski conjecture~\cite{Mi:DA,M:LP,D:NC,Co:MC,W:MC,Ma:AC1,GM:MM,GM:P2,GGM:PM}  and use 
results of its proof. Corresponding results and conjectures are stated in the simpliest way in terms of geometric lattices and critical
lattices~\cite{Mi:GZ,Mi:DA,Cassels,lek}. These lattices constitute an important particular case of geometric lattices.
We consider balls in the plane, defined by
\[   D_p: \;  |x|^p + |y|^p \le 1 \]
for different values of $p \geq 1$. We investigate lattice packings of Minkowski balls, Davis balls  and Chebyshev-Cohn balls, names coming 
in correspondence with ranges for $p$, i.e. with varying Minkowski metric, and search for optimal lattice packings of these balls. 
The packing problem is studied on  classes of lattices related to the problem of the theory of Diophantine approximations 
considered by H. Minkowski \cite{Mi:DA} for the plane case.
For some other selected  problems and results of the theory of Diophantine approximations the interested reader can 
look, for instance in \cite{AndersenDuke} and references therein.
The second section deals very briefly with shells of critical lattices of Minkowski curves, 
strictly convex and smooth algebraic curves related to the Minkowski conjecture, and integer points on curves.
Our naming of balls is related and justified with results of investigations of the Minkowski conjecture on critical determinant.
We present examples of optimal lattice packings with Minkowski, Davis and  Chebyshev-Cohn balls and domains. Our main results are
theorems on lattices realizing optimal packings by the corresponding balls and domains, which are formulated and  their proof is given.
More precisely, we prove in section \ref{secpmdccballs} that the optimal lattice packing of the Minkowski, Davis, and Chebyshev-Cohn balls is realized with respect to the sublattices of index two of the critical lattices of corresponding  balls.
In section \ref{secpmdccdomains} ``Critical lattices, Minkowski domains and  their optimal packing'' we investigate lattice packings of  Minkowski    domains. By results of the proof of Minkowski conjecture about the critical determinant we devide the domains on 3 classes: Minkowski, Davis  and  Chebyshev-Cohn.  The optimal lattice packings of the   domains are obtained.
Section \ref{secichminarea}  gives minimum areas of  inscribed and circumscribed hexagons of  Minkowski,  Davis  and  Chebyshev-Cohn domains. We construct direct systems of Minkowski, Davis and Chebyshev-Cohn balls and domains, direct systems of their critical lattices and calculate their direct limits in section \ref{secdsdld}.

\section{Minkowski conjecture, its proof and  Minkowski balls}
Let
 $$ |\alpha x + \beta y|^p + |\gamma x + \delta y|^p \leq c \cdot 
 |\det(\alpha \delta - \beta \gamma)|^{p/2}, $$ 

be a diophantine inequality defined for a given real $ p >1 $;
hear $\alpha, \beta, \gamma, \delta$ are real numbers with 
$ \alpha \delta - \beta \gamma \neq 0 .$ 

H. Minkowski in his monograph~\cite{Mi:DA} raise the question
about minimum constant $c$ 
such that the inequality has integer 
solution other than origin. 
Minkowski with the help of his theorem on convex body has found a sufficient condition for the solvability of Diophantine inequalities in integers not both zero:
$$ c = {\kappa_p}^p,  \kappa_{p} = \frac{{\Gamma(1 + \frac{2}{p})}^{1/2}}{{\Gamma(1 + \frac{1}{p})}}.$$
 But this result is not optimal, and Minkowski also raised the issue of not improving constant $c$.
For this purpose Minkowski has proposed to use the critical determinant.

Recall the definitions~\cite{Cassels}.

Let $\mathcal R $ be a set and $\Lambda $ be a  lattice with base 
\[
\{a_1, \ldots ,a_n \}
\] 
in ${\mathbb R}^n.$ \\

A lattice $\Lambda $
is {\it admissible} for body $\mathcal R $ 
($ {\mathcal R}-${\it admissible})
if ${\mathcal R} \bigcap \Lambda = \emptyset $ or $0.$\\
 
 Let $\overline {\mathcal R}$ be the closure of $ {\mathcal R}$.
A lattice $\Lambda $
is {\it strictly admissible} for  $\overline {\mathcal R}$  
( $\overline {\mathcal R}-${\it strictly  admissible})
if $\overline {\mathcal R} \bigcap \Lambda = \emptyset $ or $0.$\\

Let
\[
  d(\Lambda) = |\det (a_1, \ldots ,a_n )|
 \] 
   be the determinant of 
$\Lambda.$ \\

The infimum
$\Delta(\mathcal R) $ of determinants of all lattices admissible for
$\mathcal R $ is called {\em the critical determinant} 
of $\mathcal R; $
if there is no $\mathcal R-$admissible lattices then puts
$\Delta(\mathcal R) = \infty. $ \\

 A lattice 
$\Lambda $ is {\em critical}
if $ d(\Lambda) = \Delta(\mathcal R).$ \\

 For the given 2-dimension
region $ D_p \subset {\mathbb R}^2 = (x,y), \ p \ge 1 $ :

$$ |x|^p + |y|^p < 1 , $$
let $\Delta(D_p) $ 
be the critical determinant of the region.

 In notations
\cite{GGM:PM} 
next result have proved:

\begin{theorem}
\cite{GGM:PM}
\label{ggm}
$$\Delta(D_p) = \left\{
                   \begin{array}{lc}
    \Delta(p,1), \; 1 \le p \le 2, \; p \ge p_{0},\\
    \Delta(p,\sigma_p), \;  2 \le p \le p_{0};\\
                     \end{array}
                       \right.
                           $$
here $p_{0}$ is a real number that is defined unique by conditions
$\Delta(p_{0},\sigma_p) = \Delta(p_{0},1),  \;
2,57 < p_{0}  < 2,58, \; p_0  \approx 2.5725 $
\end{theorem}
\begin{remark}
We will call $p_{0}$ the Davis constant.
\end{remark}
\begin{corollary}
$$ {\kappa_p} = {\Delta(D_p)}^{-\frac{p}{2}}.  $$
\end{corollary}

\begin{corollary}

From theorem (\ref{ggm}) in notations \cite{GGM:PM,Gl4} we have next expressions for critical determinants and their lattices:  
\begin{enumerate} 

\item  ${\Delta^{(0)}_p} = \Delta(p, {\sigma_p}) =  \frac{1}{2}{\sigma}_{p},$ 

\item $ {\sigma}_{p} = (2^p - 1)^{1/p},$

\item  ${\Delta^{(1)}_p}  = \Delta(p,1) = 4^{-\frac{1}{p}}\frac{1 +\tau_p }{1 - \tau_p}$,   

\item  $2(1 - \tau_p)^p = 1 + \tau_p^p,  \;  0 \le \tau_p < 1.$  

\end{enumerate}
\end{corollary}
  For their critical lattices respectively  $\Lambda_{p}^{(0)},\; \Lambda_{p}^{(1)}$ next conditions satisfy:   $\Lambda_{p}^{(0)}$ and 
 $\Lambda_{p}^{(1)}$  are  two $D_p$-admissible lattices each of which contains
three pairs of points on the boundary of $D_p$  with the
property that 
\begin{itemize}

\item $(1,0) \in \Lambda_{p}^{(0)},$

\item $(-2^{-1/p},2^{-1/p}) \in \Lambda_{p}^{(1)},$

\end{itemize}
 (under these conditions the lattices are
uniquely defined).

\begin{example}
 Lattice $\Lambda_{p}^{(0)}$ are two-dimensional lattices in ${\mathbb R}^2$ spanned by the vectors
 \begin{itemize}
 \item[]  $\lambda^{(1)} = (1, 0),$
   \item[] $\lambda^{(2)} = (\frac{1}{2}, \frac{1}{2} \sigma_p).$
   \end{itemize}
   For lattices $\Lambda_{p}^{(1)}$, due to the cumbersomeness of the formulas for the coordinates of basis vectors of the
    lattices, as an example, we present only the lattice $\Lambda_{2}^{(1)}$.
    The lattice $\Lambda_{2}^{(1)}$ is a two-dimensional lattice in ${\mathbb R}^2$ spanned by the vectors
 \begin{itemize}
 \item[]  $\lambda^{(1)} = (-2^{-1/2},2^{-1/2}),$
   \item[] $\lambda^{(2)} = (\frac{\sqrt 6 - \sqrt 2}{4}, \frac{\sqrt 6 + \sqrt 2}{4}).$
   \end{itemize}
   \end{example}

\section{Minkowski curves, lattice points  and shells}

Minkowski~\cite{Mi:DA} also noted that the points of the critical lattice on the curve 
\begin{equation}
\label{mcp}
C_p: |x|^p + |y|^p = 1, p > 1
\end{equation}
form a regular hexagon. We will call (\ref{mcp}) Minkowski curves.

\subsection{Shells of critical lattices of Minkowski curves.}

\begin{lemma}
Let $(P_x, P_y)$ be a point of the critical lattice $\Lambda$ on Minkowski curve (\ref{mcp}).
Then the point $(u, v)$ that satisfies conditions
\begin{equation}
\label{es}
\left\{
                   \begin{array}{lc}
   |P_x v - P_y u| = \; d(\Lambda),\\
    |u|^p +  |v|^p = \; 1.\\
 \end{array}
                       \right.
\end{equation}
belongs to $\Lambda$ and   lies on (\ref{mcp}). 
The shell of points of the critical lattice $\Lambda_{p}$ on Minkowski curve (\ref{mcp}) contains 6 points
Point coordinates $(u, v)$ can be calculated in closed form or with any precision.
\end{lemma}
{\bf Proof}. From theorem \ref{ggm} we have the value of $d(\Lambda)$. From \cite{Mi:DA,Cassels}  and also explicitly  from (\ref{es}) follows that the shell contains 6 points.
The rest is obvious.
\begin{example}
Each shell of points of the critical lattices $\Lambda_{p}^{(0)}$ on Minkowski curves (\ref{mcp}) contain 6 points:
\[
  \pm (1, 0), 
  \]
  \[
  \left(\pm \frac{1}{2},  \pm \frac{1}{2} \sigma_p \right).
\]
  
Respectively the shell of points of the critical lattice $\Lambda_{2}^{(1)}$ on Minkowski curves (\ref{mcp}) contains 6 points:
\[
  (-2^{-1/2},2^{-1/2}), (2^{-1/2},-2^{-1/2}),  
  \]
  \[
  \pm \left(\frac{\sqrt 6 + \sqrt 2}{4}, \frac{\sqrt 6 - \sqrt 2}{4}\right), 
  \]
  \[
    \pm \left(\frac{\sqrt 6 - \sqrt 2}{4}, \frac{\sqrt 6 + \sqrt 2}{4} \right).
\]
\end{example}

\begin{remark}
The only integer quadratic forms of the lattices $\Lambda_{p}^{(0)}$ and $\Lambda_{p}^{(1)}$ are the quadratic forms corresponding to the lattices
$\Lambda_{2}^{(0)}$    and $\Lambda_{2}^{(1)}$.
These forms are the  form  
\begin{equation}
 x^2 - xy +y^2
 \end{equation}
 for the  lattice $\Lambda_{2}^{(0)}$ with the base $\lambda^{(1)} = (1, 0), \lambda^{(2)} = (-\frac{1}{2}, \frac{1}{2} \sqrt 3)$  and the form 
 \begin{equation}
 \label{ql21}
 x^2 -xy +y^2
\end{equation}
 for the  lattice $\Lambda_{2}^{(1)}$ with the base
  $\lambda^{(1)} = (-2^{-1/2},2^{-1/2}), \lambda^{(2)} = (\frac{\sqrt 6 + \sqrt 2}{4}, \frac{\sqrt 6 - \sqrt 2}{4})$ .
 \end{remark}

For these critical lattices, it is easy to  calculate from the corresponding theta function, their 
shells for the natural numbers.
 Corresponding theta series \cite{CS,Serre} for $\Lambda_{2}^{(1)}$ with
$q=\exp(\pi i \tau)$  is defined as
\begin{equation}
 \Theta_{\Lambda_{2}^{(1)}}(\tau) = \sum_{\lambda \in \Lambda_{2}^{(1)}} q^{||\lambda||_2} 
= \sum_{x, y = -\infty}^{\infty} q^{(x^2 - xy + y^2)}  = \sum_{m=0}^{\infty} N_{m} q^{m}.
\end{equation}
 where $N_m$ is the number of times (\ref{ql21}) represents $m, N_0 = 1$.
 Thus
 \[
   \Theta_{\Lambda_{2}^{(1)}}(\tau) = 1 + 6q + 6q^{3} + 6q^{4}
    + 12q^{7} +  6q^{9} + 6q^{12} +  \cdots
 \]

\subsection{Strictly convex curves and smooth algebraic curves from   moduli spaces}

For $1 < p < \infty$ Minkowski curves are strictly convex curves.
\begin{remark}
The length $\ell$ of the Minkowski curve for $p > 1$ is
 \[
 {\ell} = 4 \int_0^1 (1 - y^p)^{\frac{1}{p}} dy
   \]
 \end{remark}
 
 \begin{remark}
As $\ell > 3$ by results of the Jarnic \cite{jar} the curves $C_p$ and for integer $N > 0$ their $N$-fold magnifications 
$N\cdot C_p$ about the origin for corresponding length $\ell $ contains at most 
\[
 3(4\pi)^{-\frac{1}{3}} {\ell}^{\frac{2}{3}}  + {\mathcal O}({\ell}^{\frac{1}{3}})
\]
integral lattice points.
\end{remark}
\begin{remark}
The curve $C_p$ for real $p > 1$ is three times continuity differentiable. From results of  Swinnerton–Dyer \cite{sdy} follow that the number of rational points $(\frac{m}{N},\frac{n}{N})$ on $C_p$ (the number of integral points on $N\cdot C_p$) as $N \to \infty$ is estimatrd by
\[
 |N \cdot C_p \cap {\mathbb Z}^2| \le c(C_p, \epsilon)N^{\frac{3}{5} + \epsilon}.
\]
Let $y = f(x)$ be the arc $A$ of (\ref{mcp}) and $f(x)$ be a transcendental analytic function on $\left[0, 1\right]$.
Then   
\[
|tA \cap {\mathbb Z}^2| \le c(f,\epsilon) t^{\epsilon}.
\]
It follows from  Bombieri-Pila formula\cite{bop}.
 \end{remark}
\subsubsection{Minkowski algebraici  curves}
(\cite{GLp,Gplav,GDAV} and references therein)\\
 For natural $p, d, p = 2d$ the curves 
 \[
 C_{2d}: x^{2d} + y^{2d} = 1
 \]
  are algebraic curves and their projective models have for $d > 1$  the genus $g = 
  (2d - 1)(d - 1)$.
  \begin{remark}
    Over complex numbers ${\mathbb C}$ the polynomial $ p(x,y) = x^{2d} + y^{2d} - 1 \in {\mathbb C}[x,y]$ 
    defines the projective curve of the form $ x^{2d} + y^{2d} - z^{2d} = 0$ in the projective space
    ${\mathbb C}{\mathbb P}^2$. The solution $y$ of $p(x,y) = 0$ is a  many-valued function  of $x$ and 
    $ C_{2d}$ is its  Riemann surface.
  \end{remark}
  \begin{remark}
 For algebraic curve   $C_{2d}({\mathbb C})$ its   Riemann surface is oriented sphire with $g$ handles. 
 The function field ${\mathbb F}$ on the  Riemann surface is the field of rational functions of the genus $g$. 
If two such fields are isomorphic then the corresponding algebraic curves are birationaly isomorphic.
   \end{remark}
   Let $K$ be the field of algebraic numbers and  $C_{2d}(K)$ be the projective curve over $K$. Let $\mathcal A$ be an analytic subset of $C_{2 d}({\mathbb C})$.
   By results of Bombieri-Pila \cite{bop} in the form of Gasbarri \cite{gas} we have:
  \begin{remark} 
    Let $(C_{2d}, {\mathcal L})$ be a projective variety over $K$ with ample line bundle ${\mathcal L}$ over $C_{2d}$ and let 
    $h_{\mathcal L}(\cdot)$ be a height function assosiated to it.
    Let $\mathcal A$ be a non compact Riemann surface. For relatively compact open set $U \subset {\mathcal A}$ and for
    holomorphic map $\varphi: \mathcal A \to  C_{2 d}({\mathbb C})$ with Zariski dense  set $\varphi(\mathcal A)$ and the set 
    $S_U(T) = \{ z \in U | \varphi(z) \in  C_{2 d}(K)\}$ and $h_{\mathcal L}(\varphi(z)) \le T$ for cardinality $\#S_U(T)$of the set 
    $S_U(T)$ and for any positive number $\epsilon$ we have 
    \[
    \#S_U(T) \ll \exp (\epsilon T)
    \]
    where the involved constants depend on $U, \varphi$,  and ${\mathcal L}$ but not on $T$ .
   \end{remark}

  Another class of algebraic curves, namely, the class of elliptic curves, arises from the complexification of admissible lattices.
These lattices form the class of admissible lattices $\Lambda$ having three pairs of points on the boundary $C_p$  of 2-dimension
region $ D_p$.
  We will consider this case elsewhere.
  

\section{Minkowski balls and the density  of the packing of $2$-dimensional unit Minkowski balls}
\label{Minkowski balls}
Recall the Davis constant $p_{0} \in {\mathbb R}$, satisfying $2,57 < p_{0}  < 2,58$. 
We consider balls of the form
\[
   D_p: \;  |x|^p + |y|^p \le 1, \; p \ge 1,
\]
and call such balls with $1<p<2$ {\it Minkowski balls} and, correspondingly, the circles defined by
\[ |x|^p + |y|^p = 1, \; 1 < p < 2 \]

are called {\it Minkowski circles}. Continuing this, we consider the following classes of balls and circles.
\begin{itemize}
  
\item  {\it Limit Minkowski circles}:    $|x| + |y| = 1$;

\item {\it Davis balls}: $|x|^p + |y|^p \le 1$  for $p_{0} > p \ge 2$;

\item {\it Davis circles}:  $|x|^p + |y|^p = 1$  for $p_{0} > p \ge 2$;

\item {\it Chebyshev-Cohn balls}: $|x|^p + |y|^p \le 1$  for $ p \ge p_{0}$;

\item {\it Chebyshev-Cohn circles}: $|x|^p + |y|^p = 1$  for $ p \ge p_{0}$;

\item {\it Limit Chebyshev-(Cohn) balls}: $||x, y||_{\infty} = \max(|x|,|y|)$.
 
\end{itemize} 

Recall the definition of a packing lattice \cite{via,ckmrv,Cassels,lek}.
We will give it for $n$-dimensional Minkowski balls $D_p^n$ in ${\mathbb R}^n$.

\begin{definition}
 Let $\Lambda$ be a full lattice in ${\mathbb R}^n$ and $a \in {\mathbb R}^n$. 
 In the case if it is occurs that no two  of balls $\{D_p^n + b, b \in \Lambda + a \}$ have inner points in common, the collection of balls 
   $\{D_p^n + b, b \in \Lambda + a \}$ is called a $(D_p^n,  \Lambda)$-packing, and $\Lambda$ is called a packing lattice of $D_p^n$.
\end{definition}

Recall also that if $\alpha \in {\mathbb R}$ and $D_p^n$ is a ball than $\alpha D_p^n$ is the set of points $\alpha x, x \in D_p^n$.
 
In some cases we will consider interiors of balls $D_p =  D_p^2$ (open balls) which we will denoted as $ID_p$.

Denote by $V(D_p)$ the volume (area) of $D_p$.

\begin{proposition} 
 The volume of  Minkowski ball $D_p$ is equal $4 \frac{(\Gamma(1 + \frac{1}{p}))^2}{\Gamma(1 + \frac{2}{p})}$.
\end{proposition}
\begin{proof} (by Minkowski). Let $x^P + y^p \le 1, x \ge 0, y \ge 0.$
Put $x^p = \xi, y^p = \eta.$
\begin{equation}
 \label{vdp}
V(D_p) = \frac{4}{p^2}\iint \xi^{\frac{1}{p} - 1}\eta^{\frac{1}{p} - 1} d\xi d\eta,
\end{equation}
where the integral extends to the area
\[
  \xi + \eta \le 1, \; \xi \ge 0, \; \eta \ge 0.
\]
 Expression (\ref{vdp}) can be represented in terms of Gamma functions, and we get
\[
 V(D_p) = 4 \frac{(\Gamma(1 + \frac{1}{p}))^{2}}{\Gamma(1 + \frac{2}{p})}.
\]
\end{proof}

From the considerations of Minkowski and other authors \cite{Mi:DA,Mi:GZ,Cassels,lek,Gl4}, the following statements can be deduced 
(for the sake of completeness, we present the proof of Proposition \ref{p1}).\\

\begin{proposition} \cite{Cassels,lek}.
\label{p1}
A lattice $\Lambda$ is a packing lattice of $D_p$ if and only if  it is admissible lattice for $2 D_p$.
\end{proposition}
\begin{proof}  (contrary proof). First, note that one can take an open ball $ID_p$ and use the notion of strict admissibility.
 Suppose that $\Lambda$ is not strictly admissible for $2 ID_p$. Then $2 ID_p$ contains a point $a \ne 0$ of $\Lambda$.
Then the two balls $ID_p$ and $ID_p +a$ contain the point $\frac{1}{2} a$  in common. So $\Lambda$ is not a packing
 lattice of $D_p$.\\
 Suppose now that $\Lambda$ is not a packing  lattice of $D_p$. Then there exist two distinct points $b_1, b_2 \in \Lambda$ and a point  $c$ such that $c \in ID_p + b_1 $ and $c \in ID_p + b_2 $. 
 Hence there are points $a_1, a_2 \in ID_p$ such that $c = a_1 + b_1 = a_2 + b_2$. So $b_1 - b_2 = a_2 - a_1 \in ID_p$, whereas 
 $b_1 - b_2 \ne 0$ and $b_1 - b_2 \in \Lambda$. Therefore $\Lambda$ is not (strictly) admissible lattice.    
\end{proof}

\begin{proposition} \cite{Cassels,lek}.
\label{p3}
 The dencity of a $(D_p, \Lambda)$-packing is equal to $V(D_p)/d(\Lambda)$ and it is maximal if $\Lambda$ is critical for  $2 D_p$.
\end{proposition}

\section{On packing  Minkowski, Davis and  Chebyshev-Cohn unit balls on the plane}
\label{secpmdccballs}
Let us consider possible optimal lattice packings of these balls and their connection with critical lattices.

At first give lattices of trivial optimal lattice packings for the limiting  (asymptotic) cases at the points $p = 1$ and $\infty$ "infinity" (the latter corresponds to the classical Chebyshev balls) and as the introductory example the optimal packing of two-dimensional unit balls.

\begin{proposition}
\label{lp2}
The lattice 
\[
  \Lambda_{1}^{(1)} =  \{(\frac{1}{2}, \frac{1}{2}), (0, 1)\}
\]
is the critica lattice for $D_1$. 
Limiting case of Minkowski balls  for $ p = 1$ gives the optimal sublattice of index two of the lattice $\Lambda_{1}^{(1) }$-lattice  packing with the density $1$.
The centers of the Minkowski balls in this case are  at the vertices of the sublattice of index two of the lattice $\Lambda_{1}^{(1) }$.
\end{proposition}
{\bf Proof.} Recall that a critical lattice for $D_1$ is a lattice $\Lambda$ which is $D_1$-admissible and which has determinant 
$d(\Lambda) = \Delta(D_1)$. The lattice $\Lambda_{1}^{(1)}$ is $D_1$-admissible.   We have $\Delta(D_1) = \frac{1}{2}$ and $d(\Lambda_{1}^{(1)}) = \frac{1}{2}$. Minkowski balls  for $ p = 1$ are congruent squares. Hance we have the optimal 
the sublattice of index two of the lattice $\Lambda_{1}^{(1) }$  packing of the squares with the density $1$.

  \begin{proposition}
\label{lp4}
  The lattice 
\[
  \Lambda_{\infty}^{(1)} =  \{(1, 1), (0, 1)\}
 \]
is the critica lattice for $D_{\infty}$. 
Limiting case of Minkowski balls  for $ p = \infty$ gives  the optimal  of the  density $1$ packing 
 with respect to the sublattice of index two of the critical lattice $ \Lambda_{\infty}^{(1)}$.
 The centers of the Minkowski balls in this case are at the vertices of  the sublattice of index two of the lattice $\Lambda_{\infty}^{(1)}$.
  \end{proposition}
{\bf Proof.}    The lattice $\Lambda_{\infty}^{(1)}$ is $D_{\infty}$-admissible.   We have $\Delta(D_{\infty}) = 1$ and
 $d(\Lambda_{\infty}^{(1)} = 1$. Minkowski balls  for $ p = \infty$ are congruent squares. Hance we have the optimal 
 sublattice of index two of the lattice $\Lambda_{\infty}^{(1)}$  packing of the squares with the density $1$.\\

\begin{proposition}
\label{lp3}
The lattice 
\[
\Lambda_{2}^{(0)} =  \{(1, 0), (\frac{1}{2}, \frac{\sqrt{3}}{2})\}
\]
is the critica lattice for $D_2$. 
The lattice 
packing of Davis balls  for $p = 2$ gives the optimal  of the  density $\approx 0.91$ packing 
 with respect to the sublattice of index two of the critical lattice $\Lambda_{2}^{(0)}$.
 The centers of the Minkowski balls in this case are at the vertices of  the sublattice of index two of the lattice $\Lambda_{2}^{(0)}$.
\end{proposition}

{\bf Proof.} As in Proposition (\ref{lp2}) a critical lattice for $D_2$ is a lattice $\Lambda$ which is $D_2$-admissible and which has determinant 
$d(\Lambda) = \Delta(D_2)$. The lattice $\Lambda_{2}^{(0)}$ is $D_2$-admissible.  We have $\Delta(D_2) = \frac{\sqrt{3}}{2}$ and $d(\Lambda_{2}^{(0)}) = \frac{\sqrt{3}}{2}$. So sublattice of index two of the lattice $\Lambda_{2}^{(0)} $ is 
 the hexagonal lattice.  Next, we use the following classical results \cite{FegesToth,Thue}: the optimal  sphere packing of dimension 2  is the  hexagonal lattice (honeycomb) packing with the density $\approx 0.91$.
  
\begin{proposition}
\label{p2}
If $\Lambda$ is the critical lattice of the Minkowski ball $D_p$ than the sublattice $\Lambda_2$  of index two of the critical lattice is the critical lattice of $2 D_p$.
 (Examples for $ n = 1, 2, \infty $ above). 
\end{proposition}
{\bf Proof.} Since the Minkowski ball $D_p$ is symmetric about the origin and convex, then $2D_p$ is convex and symmetric
 about the origin \cite{Cassels,lek}.
 
When parametrizing admissible lattices $\Lambda$ having three pairs of points on the boundary of the ball $D_p$,
the following parametrization is used \cite{Co:MC,GM:P2,GGM:PM,Gl4}:
 \begin{equation}
\label{eq2}
\Lambda =  \{((1 + \tau^{p})^{-\frac{1}{p}}, \tau(1 + \tau^{p})^{-\frac{1}{p}}), (-(1 + \sigma^p)^{-\frac{1}{p}}, \sigma(1 + \sigma^p)^{-\frac{1}{p}})\}
\end{equation}
 where 
 \[
 0 \le  \tau < \sigma , \; 0 \le \tau \le \tau_p.
  \]
   $\tau_{p}$ is defined by the
equation $ 2(1 - \tau_{p})^{p} = 1 + \tau_{p}^{p}, \; 0 \leq
\tau_{p} < 1. $ 
\[
  1 \le\sigma \le \sigma_p, \; \sigma_p = (2^p - 1)^{\frac{1}{p}}.
\]
Admissible lattices of the form (\ref{eq2}) for doubled Minkowski balls $2 D_p$ have a representation of the form 
 \begin{equation}
\label{eq3}
\Lambda_{2 D_p} =  \{2((1 + \tau^{p})^{-\frac{1}{p}}, 2\tau(1 + \tau^{p})^{-\frac{1}{p}}), (-2(1 + \sigma^p)^{-\frac{1}{p}}, 2\sigma(1 + \sigma^p)^{-\frac{1}{p}})\}
\end{equation}
  Hence the Minkowski-Cohn moduli space for these admissible lattices has the form
 \begin{equation}
 \Delta(p,\sigma)_{2 D_p} = 4 (\tau + \sigma)(1 + \tau^{p})^{-\frac{1}{p}}
  (1 + \sigma^p)^{-\frac{1}{p}}, 
 \end{equation}
in the same domain
 $$ {\mathcal M}: \; \infty > p > 1, \; 1 \leq \sigma \leq \sigma_{p} =
 (2^p - 1)^{\frac{1}{p}}, $$
  
  \begin{corollary}
 Consequently, the critical determinants of doubled Minkowski balls have a representation of the form

\begin{enumerate}  

\item ${\Delta^{(0)}_p}(2D_p) = \Delta(p, {\sigma_p})_{2D_p} =  2\cdot {\sigma}_{p},$

\item  ${\sigma}_{p} = (2^p - 1)^{1/p},$

\item ${\Delta^{(1)}_p}(2D_p)  = \Delta(p,1))_{2D_p} = 4^{1 - \frac{1}{p}}\frac{1 +\tau_p }{1 - \tau_p},$

\item  $ 2(1 - \tau_p)^p = 1 + \tau_p^p,  \;  0 \le \tau_p < 1. $

\end{enumerate}
\end{corollary}

  And these are the determinants of the sublattices of index 2 of the critical lattices of the corresponding Minkowski balls.
  
   \begin{theorem} 
  \label{opb} 
The optimal lattice packing of  Minkowski and Chebyshev-Cohn balls is realized with respect to the sublattices of the index two of  critical lattices
  \[
  (-2^{-1/p},2^{-1/p}) \in \Lambda_{p}^{(1)}.
  \]   
The dencity of this optimal packing is given by the expression
\[
    V(D_p)/\Delta^{(1)}_p (2D_p),
\]
where $V(D_p)$ is the volume of the ball $D_p$ and $\Delta^{(1)}_p (2D_p)$ is the critical determinant of the domain $2D_p$.\\
The optimal lattice packing of   Davis  balls is realized with respect to the sublattices of the index two of  critical lattices
  \[
  (1,0)\in\Lambda_{p}^{(0)},
  \]   
The dencity of this optimal packing is given by the expression
\[
    V(D_p)/\Delta^{(0)}_p(2D_p)
\]
where $V(D_p)$ is the volume of the ball $D_p$ and $\Delta^{(0)}_p(2D_p)$ is the critical determinant of the domain $2D_p$.
  \end{theorem}
{\bf  Proof}.
 By Proposition \ref{p2} the critical lattice of $2 D_p$ is  the sublattice  of index two of the critical lattice 
of Minkowski ball $D_p$ .
So it is the admissible lattice for $2 D_p$ and by Proposition \ref{p1} 
is packing lattice of $D_p$.
By Proposition \ref{p3} the corresponding lattice packing has maximal density and so is optimal.
  
\begin{corollary}
 For the optimal lattice packing of  Minkowski and Chebyshev-Cohn balls we have
$ \Delta^{(1)}_p (2D_p) = 4^{1 - \frac{1}{p}} \frac{1 + \tau_p}{1 - \tau_p}$, so the central density is
\[
   \delta = 4^{\frac{1}{p} - 1} \frac{1 - \tau_p}{1 + \tau_p}.
\]
 For the optimal lattice packing of   Davis  balls we have
$\Delta^{(0)}_p (2D_p) = 2 \sigma_p$, so the central density is
\[
   \delta = \frac{1}{2} {\sigma_p}^{-1}.
\]
\end{corollary}
   
   \section{Critical lattices, Minkowski domains and  their optimal packing}
   \label{secpmdccdomains}

   Let $D$ be a fixed bounded symmetric about origin  convex body ({\it centrally symmetric convex body} for short) with volume
$V(D)$.
When considering packing problems for such $D$, it does not matter whether we consider such $D$ with or without a boundary \cite{Cassels,lek}.

\begin{proposition} \cite{Cassels,lek}.
\label{p1}
 If $D$ is symmetric about the origin and convex, then $2D$ is convex and symmetric
 about the origin.
\end{proposition}

 \begin{corollary}
 \label{cor1}
 Let $m$ be integer $m \ge 0$ and $n$ be natural greater $m$. 
If $2^m D$ centrally symmetric convex body then $2^n D$ is again centrally symmetrc convex body.
\end{corollary}
{\bf Proof.} Induction. \\

\subsection{Domains.}

We consider the following classes of balls (see Section \ref{Minkowski balls}) and domains.

\begin{itemize}
  
\item  {\it  Minkowski domains}:    $2^m D_p$,  integer $m \ge 1$, for $1 \le p<2$;

\item {\it Davis domains}: $2^m D_p$,  integer $m \ge 1$, for $p_{0} > p \ge 2$;


\item {\it Chebyshev-Cohn domains}: $2^m D_p$,  integer $m \ge 1$,  for $ p \ge p_{0}$;

\end{itemize}

\begin{remark}
\label{r2}
Sometimes, when it comes to using $p$ values that include scopes of different domains, we will use the term Minkowski domains for definitions of different types of domains, specifying
  name when a specific $p$ value or a range of $p$ values containing a single type of domains is specified.
\end{remark}  

\begin{proposition}
\label{p7}
Let $m$ be  integer, $m \ge 1$.
If $\Lambda$ is the critical lattice of the 
ball $D_p$ than the sublattice $\Lambda_{2^m}$  of index $2^m$ is the critical lattice of the domain $2^{m-1} D_p$.
\end{proposition}
{\bf Proof.} Since the Minkowski ball $D_p$ is symmetric about the origin and convex, then $2^m D_p$ is convex and symmetric
 about the origin by Corollary \ref{cor1}.
 
When parametrizing admissible lattices $\Lambda$ having three pairs of points on the boundary of the ball $D_p$,
the following parametrization is used \cite{Co:MC,GGM:PM,Gl4}:
 \begin{equation}
\label{eq2}
\Lambda =  \{((1 + \tau^{p})^{-\frac{1}{p}}, \tau(1 + \tau^{p})^{-\frac{1}{p}}), (-(1 + \sigma^p)^{-\frac{1}{p}}, \sigma(1 + \sigma^p)^{-\frac{1}{p}})\}
\end{equation}
 where 
 \[
 0 \le  \tau < \sigma , \; 0 \le \tau \le \tau_p.
  \]
   $\tau_{p}$ is defined by the
equation $ 2(1 - \tau_{p})^{p} = 1 + \tau_{p}^{p}, \; 0 \leq
\tau_{p} < 1. $ 
\[
  1 \le\sigma \le \sigma_p, \; \sigma_p = (2^p - 1)^{\frac{1}{p}}.
\]
Admissible lattices of the form (\ref{eq2}) for $2^m$ {\it  dubling} Minkowski domains $2^m D_p$ have a representation of the form 
 \begin{equation}
\label{eq3}
\Lambda_{2^m D_p} =  \{2^m ((1 + \tau^{p})^{-\frac{1}{p}}, 2^m \tau(1 + \tau^{p})^{-\frac{1}{p}}), (-2^m (1 + \sigma^p)^{-\frac{1}{p}},      2^m \sigma(1 + \sigma^p)^{-\frac{1}{p}})\}
\end{equation}
  Hence the Minkowski-Cohn moduli space for these admissible lattices has the form
 \begin{equation}
 \Delta(p,\sigma)_{2^m D_p} = 2^{2m} (\tau + \sigma)(1 + \tau^{p})^{-\frac{1}{p}}
  (1 + \sigma^p)^{-\frac{1}{p}}, 
 \end{equation}
in the same domain
 $$ {\mathcal M}: \; \infty > p > 1, \; 1 \leq \sigma \leq \sigma_{p} =
 (2^p - 1)^{\frac{1}{p}}, $$

Consequently, the critical determinants of $2^m$ {\it  dubling}
  balls $D_p$ have a representation of the form
\begin{equation}
\label{cdde0}
 {\Delta^{(0)}_p}(2^m D_p) = \Delta(p, {\sigma_p})_{2^m D_p} =  2^{m-1}\cdot {\sigma}_{p},   \; {\sigma}_{p} = (2^p - 1)^{1/p},
 \end{equation}
 \begin{equation}
\label{cdde1}
 {\Delta^{(1)}_p}(2^m D_p)  = \Delta(p,1))_{2^m D_p} = 4^{m - \frac{1}{p}}\frac{1 +\tau_p }{1 - \tau_p},   \;  2(1 - \tau_p)^p = 1 + \tau_p^p,  \;  0 \le \tau_p < 1. 
 \end{equation}
  And these are the determinants of the sublattices of index $2^m$ of the critical lattices of the corresponding  balls $D_p$.
  
\begin{remark}
  Proposition \ref{p7} is a strengthening of Proposition \ref{p2} and its extension on  domains.  
\end{remark}

\begin{proposition} \cite{Cassels,lek}.
\label{p2m21m}
 The dencity of a $(2^{m-1} D_p, \Lambda)$-packing is equal to $V(2^{m-1} D_p)/d(\Lambda)$ and it is maximal if $\Lambda$ is critical for  $(2)^m D_p$.
\end{proposition}

\begin{theorem}
\label{opd}
The optimal lattice packing of  Minkowski and Chebyshev-Cohn domains $2^m D_p$ is realized with respect to the sublattices of the index $2^{m+1}$ of  critical lattices
  \[
  (-2^{-1/p},2^{-1/p}) \in \Lambda_{p}^{(1)}.
  \]   
The dencity of this optimal packing is given by the expression
\[
    V(2^m D_p)/\Delta^{(1)}_p (2^{m+1}D_p),
\]    
    where $V(2^m D_p)$ is the volume of the domainl $2^m D_p$ and $\Delta^{(1)}_p (2^{m+1} D_p)$ is the critical determinant of the domain $2^{m+1} D_p$.\\
The optimal lattice packing of  Davis domains $2^m D_p$ is realized with respect to the sublattices of the index $2^{m+1}$  of  critical lattices
  \[
  (1,0)\in\Lambda_{p}^{(0)},
  \]   
The dencity of this optimal packing is given by the expression 
\[
    V(2^m D_p )/\Delta^{(0)}_p(2^{m+1} D_p)
\]
 where $V(2^m D_p)$ is the volume of the domainl $2^m D_p$ and $\Delta^{(0)}_p (2^{m+1} D_p)$ is the critical determinant of the domain $2^{m+1} D_p$.\\
     \end{theorem}
{\bf  Proof}.
 By Proposition \ref{p2} the critical lattice of $2^{m-1} D_p, m = 1, 2, \ldots$ is  the sublattice  of index $2^m$ of the critical lattice 
of the ball $D_p$ .
So it is the admissible lattice for $2^m D_p$ and by Proposition \ref{p1} 
is packing lattice of $2^{m-1} D_p$.
By Propositions \ref{p3}, \ref{p2m21m} the corresponding lattice packing has maximal density and so is optimal.

\section{Inscribed and circumscribed hexagons of minimum areas}
\label{secichminarea}

Denote by $\Delta (2^m D_p)$ the critical determinant of the domain $2^m D_p$.
By formulas (\ref{cdde0},\ref{cdde1}) we have:
\begin{equation}
\label{cdd}
\Delta(2^m D_p) = \left\{
                   \begin{array}{lc}
     {\Delta^{(0)}_p}(2^m D_p), \; 1 \le p \le 2, \; p \ge p_{0},\\
    {\Delta^{(1)}_p}(2^m D_p), \;  2 \le p \le p_{0};\\
                     \end{array}
                       \right.
 \end{equation}
here $p_{0}$ is a real number that is defined unique by conditions
$\Delta(p_{0},\sigma_p) = \Delta(p_{0},1),  \;
2,57 < p_{0}  < 2,58, \; p_0  \approx 2.5725 $\\

Denoted by  $Ihma_{2^m D_p}$ the minimal area of  hexagons which inscibed in the domain $2^m D_p$  and have three pairs of points on the boundary of $2^m D_p$. 

\begin{theorem}
\label{ihma}
\[
  Ihma_{2^m D_p} = 3 \cdot \Delta (2^m D_p).
\]
 \end{theorem}
 {\bf  Proof}. Follow from Theorems \ref{ggm}, \ref{opd} and \cite{lek}.\\
 
 Respectively denoted by  $Shma_{2^m D_p}$ the minimal area of  hexagons which circumscribed to the domain $2^m D_p$  and have three pairs of points on the boundary of $2^m D_p$. 

\begin{theorem}
\label{chma}
\[
        Shma_{2^m D_p}   = 4 \cdot \Delta (2^m D_p).
        \]
 \end{theorem} 
{\bf  Proof}.  Follow from Theorems \ref{ggm}, \ref{opd} and \cite{lek}.

\section{ Direct systems,  direct limits and duality}
\label{secdsdld}

Various variants of direct systems and direct limits are introduced and studied in \cite{Pontryagin,bur,am}.\\
Here we give their simplified versions sufficient for our purposes.
Let $X$ be a set. By a binary relation over $X$ we understand a subset of the Cartesian product $X \times X$.
By a preorder on $X$ we understand a binary relation over $X$ that is reflexive and transitive.

\begin{definition}
\label{dset}
 Let ${\mathbb N}_0$ be the set of natural numbers with zero.
A preoder $N$ on ${\mathbb N}_0$  is called a {\it directed set} if for each pair $k, m \in N$ there exists an $n \in N$ for which 
$k \le n$ and $m \le n$.
A subset  $N'$ is {\it cofinal} in  $N$ if, for each $m \in N$ their exists an $n \in N'$ such that $m \le n$.
\end{definition}

\begin{definition}
\label{dds}
A {\it direct} (or {\it inductive}) {\it system of sets} $\{X, \pi\}$ over a directed set $N$ is a function which attaches to each 
$m \in N$ a set $X^m$, and to each pair $m, n$  such that $m \le n$ in $N$, a map $\pi_m^n: X^m \to X^n$ such that, for each 
$m \in N$,  $\pi_m^m = Id$, and for $m \le n \le k$ in $N$, $\pi_n^k \pi_m^n = \pi_m^k$.
\end{definition}

\begin{remark}
We will consider direct systems of sets, topological spaces, groups and free ${\mathbb Z}$-modules.
\end{remark}

 \begin{proposition}
A directed set $N$ forms a category with elements are natural numbers $n, m, \ldots$  with zero and with morphisms
 $m \to n$ defined by the relation $m \le n$.~ A direct system over $N$ is a covariant functor from $N$ to the category of sets and maps, or to the category of topological spaces and continuous mappins, or to the category of groups  and homomorphisms or to the category of  ${\mathbb Z}$-modules and homomorphisms.
 \end{proposition}
 {\bf  Proof}. Obviously.

\begin{definition}
\label{dsm}
Let $\{X, \pi\}$ and $\{Y, \psi\}$ be direct systems over $M$ and $N$ respectively. Then a map 
\[
  \Phi:   \{X, \pi\}  \to \{Y, \psi\}
\] 

consisting of a map $\phi: M \to N$, and for each $m \in M$, a map
\[
 {\phi}^m:  X^m \to Y^{{\phi}(m)}
\]
 such that, if $m \le n$ in $M$, then commutativity holds in the diagram
\[
 \begin{CD}
   X^m @>\pi>> X^n\\
   @VV{\phi}V          @VV{\phi}V\\
   Y^{\phi(m)} @>{\psi}>>    Y^{\phi(n)}
 \end{CD}
\]

\end{definition}

\subsection{ Direct systems of Minkowski domains and their limits}

Here we will consider direct systems of Minkowski balls and domains as well as direct systems  of critical lattices.
 We use the notation according to Remark \ref{r2}.

The direct system of Minkowski balls and domains has the form (\ref{dsm}), where the multiplication by $2$  is the continuous mapping
\begin{equation}
\label{dsm}
 \begin{CD}
 D_p @>2>> 2 D_p @>2>> 2^2 D_p @>2>> \cdots @>2>> 2^m D_p @>2>> \cdots 
 \end{CD}
\end{equation}

The direct system of  critical lattices  has the form (\ref{dsml}), where the multiplication by $2$  is the homomorphism of abelian groups
\begin{equation}
\label{dsml}
 \begin{CD}
 \Lambda_{p} @>2>> 2 \Lambda_{p} @>2>> 2^2 \Lambda_{p} @>2>> \cdots @>2>> 2^m \Lambda_{p} @>2>> \cdots 
 \end{CD}
\end{equation}

Note first that real planes ${\mathbb R}^2 = (x,y)$ with Minkowski norms
\[
   |x|^p + |y|^p, \; p \in {\mathbb R}, p \ge 1,
\]
 are Banach spaces ${\mathbb B}^2_p$ .
By (general) Minkowski balls we  mean (two-dimensional) balls in ${\mathbb B}^2_p$ of the form
  \begin{equation}
   \label{eq1}
   D_p: \;  |x|^p + |y|^p \le 1, \; p \ge 1. 
  \end{equation}. 
  
 By Pontryagin duality \cite{Pontryagin} every Banach space considered in its additive structure which is an abelian topological group $G$ is reflexive that means  the existance of topological isomorphism bitween $G$ and its bidual $G^{\wedge \wedge}$.

       In our considerations we have direct systems of Minkowski balls, Minkowski domains and direct systems of critical lattices with respective maps and homomorphisms. Consider system  (\ref{dsm}) and its direct limit from the point of view of metric geometry. 
System (\ref{dsm}) is a direct system of neighborhoods of zero, and its direct limit is a 
neighborhood of zero constructed in accordance with the rules for constructing 
the direct limit of the corresponding direct system of topological spaces.
We denote this direct limits by $D^{dirlim}_p$. Accordingly, we denote the direct limits of critical lattices by $\Lambda_{p}^{dirlim}$.

\subsubsection{Direct systems}

We consider the following classes of direct systems of Minkowski balls and domains and their critical lattices.

\begin{itemize}
  
\item  {\it  Minkowski direct systems of balls and domains}:   intr $m \ge 1$, for $1 \le p<2$;
\begin{equation}
\label{dsmbd}
 \begin{CD}
 D_p @>2>> 2 D_p @>2>> 2^2 D_p @>2>> \cdots @>2>> 2^m D_p @>2>> \cdots 
 \end{CD}
\end{equation}

\item {\it  Minkowski direct systems of  critical lattices}:   int $m \ge 1$, for $1 \le p<2$;
\begin{equation} 
\label{dsmcl}
 \begin{CD}
 \Lambda_{p} @>2>> 2 \Lambda_{p} @>2>> 2^2 \Lambda_{p} @>2>> \cdots @>2>> 2^m \Lambda_{p} @>2>> \cdots 
 \end{CD}
\end{equation}

\item {\it Davis direct systems of  balls and domains}:   int $m \ge 1$, for $p_{0} > p \ge 2$;
\begin{equation}
\label{dsdbd}
 \begin{CD}
 D_p @>2>> 2 D_p @>2>> 2^2 D_p @>2>> \cdots @>2>> 2^m D_p @>2>> \cdots 
 \end{CD}
\end{equation}

\item {\it Davis direct systems of  critical lattices}:   int $m \ge 1$, for $p_{0} > p \ge 2$;
\begin{equation}
\label{dsdcl}
 \begin{CD}
 \Lambda_{p} @>2>> 2 \Lambda_{p} @>2>> 2^2 \Lambda_{p} @>2>> \cdots @>2>> 2^m \Lambda_{p} @>2>> \cdots 
 \end{CD}
\end{equation}

Below, for brevity, we will denote the balls, domains, and lattices of Chebyshev-Cohn by CC.

\item {\it CC direct systems of  balls and domains}:  int $m \ge 1$,  for $ p \ge p_{0}$;
\begin{equation}
\label{dsccbd}
 \begin{CD}
 D_p @>2>> 2 D_p @>2>> 2^2 D_p @>2>> \cdots @>2>> 2^m D_p @>2>> \cdots 
 \end{CD}
\end{equation}

\item {\it CC direct systems of  critical lattices}:  int $m \ge 1$,  for $ p \ge p_{0}$;
\begin{equation}
\label{dscccl}
 \begin{CD}
 \Lambda_{p} @>2>> 2 \Lambda_{p} @>2>> 2^2 \Lambda_{p} @>2>> \cdots @>2>> 2^m \Lambda_{p} @>2>> \cdots 
 \end{CD}
\end{equation}

\end{itemize}

\subsubsection{Direct limits}

Let us calculate direct limits of these derect systems.
Let ${\mathbb Q}_2$ and ${\mathbb Z}_2$ be respectively the field of $2$-adic numbers and its ring of integers.

 \begin{proposition}
  $D^{dirlim}_p = \varinjlim  2^m D_p  \in ({\mathbb Q}_2 / {\mathbb Z}_2)  D_p =
   (\bigcup_m \frac{1}{2^m} {\mathbb Z}_2/{\mathbb Z}_2) D_p.$
 \end{proposition}
 
  {\bf  Proof}. Follow from properties of direct systems and their direct limits \cite{bur,stei}.

 \begin{proposition}
  $\Lambda_{p}^{dirlim} = \varinjlim  2^m \Lambda_{p}  \in ({\mathbb Q}_2 / {\mathbb Z}_2)  \Lambda_{p} = 
  (\bigcup_m \frac{1}{2^m} {\mathbb Z}_2/{\mathbb Z}_2) \Lambda_{p}.$
 \end{proposition}
 
 {\bf  Proof}. Follow from properties of direct systems and their direct limits \cite{bur,stei}.\\
 

\thanks{{\bf Acknowledgments} 
The author would like to warmly thank Professor P. Boyvalenkov for his encouragement and helpful 
suggestions in the preperation of the author's paper and for comments and remarks. 
 The author is deeply grateful to the Bulgarian Academy of Sciences, the Institute of Mathematics and Informatics of the Bulgarian Academy of Sciences for their support. }

\bibliographystyle{amsplain}

\end{document}